\documentclass[11pt]{article}
\usepackage[left=1in,top=1in,right=1in,bottom=1in,head=.1in,nofoot]{geometry}

\usepackage{setspace,url,bm,amsmath} 
\usepackage{float}
\usepackage{caption}
\usepackage{subcaption,siunitx,booktabs}

\usepackage{titlesec} 
\usepackage{rotating}
\titlelabel{\thetitle.\quad} 
\titleformat*{\section}{\bf\large}

\usepackage{graphicx} 
\usepackage{bbm}
\usepackage{latexsym}
\usepackage{color}

\usepackage{amsthm}
\usepackage{amsfonts}
\theoremstyle{definition}
\newtheorem{theorem}{Theorem}

\newtheorem{lemma}{Lemma}
\newtheorem{remark}{Remark}

\usepackage{natbib} 
\bibpunct{(}{)}{;}{a}{}{,} 

\usepackage{etoolbox} 
\apptocmd{\sloppy}{\hbadness 10000\relax}{}{} 

\newcommand{\E}{\mathbb{E}}
\newcommand{\var}{\textnormal{Var}}
\newcommand{\cov}{\textnormal{Cov}}

\usepackage[toc,page]{appendix}

\usepackage{hyperref}

\begin{document}


\doublespacing
\title{\bf Sampling-based randomized designs for causal inference under the potential outcomes framework}
\author{Zach Branson\footnote{Corresponding author: 1 Oxford Street, Cambridge, MA 02138, zbranson@g.harvard.edu} \ and Tirthankar Dasgupta\thanks{We would like to thank Marielle Remillard and Chad Vecitis for conversations and collaborative work that originally inspired this work. We would also like to thank an anonymous reviewer and the editor for insightful comments that led to improvements in this work. This research was supported by the National Science Foundation under Grant No. 1144152 and Grant No. DMS 1612901. Any opinions, findings, and conclusions or recommendations expressed in this material are those of the authors and do not necessarily reflect the views of the National Science Foundation.}}

\date{}
\maketitle

\begin{abstract}
\frenchspacing
\noindent
We establish the inferential properties of the mean-difference estimator for the average treatment effect in randomized experiments where each unit in a population is randomized to one of two treatments and then units within treatment groups are randomly sampled. The properties of this estimator are well-understood in the experimental design scenario where first units are randomly sampled and then treatment is randomly assigned, but not for the aforementioned scenario where the sampling and treatment assignment stages are reversed. We find that the inferential properties of the mean-difference estimator under this experimental design scenario are identical to those under the more common sample-first-randomize-second design. This finding will bring some clarifications about sampling-based randomized designs for causal inference, particularly for settings where there is a finite super-population. Finally, we explore to what extent pre-treatment measurements can be used to improve upon the mean-difference estimator for this randomize-first-sample-second design. Unfortunately, we find that pre-treatment measurements are often unhelpful in improving the precision of average treatment effect estimators under this design, unless a large number of pre-treatment measurements that are highly associative with the post-treatment measurements can be obtained. We confirm these results using a simulation study based on a real experiment in nanomaterials.
\end{abstract}


\section{Introduction: The Potential Outcomes Framework for Randomized Experiments} \label{s:intro}

In most two-armed randomized experiments conducted to compare the causal effects of two treatments (or one treatment and one control), a finite population of $N$ experimental units is considered. These $N$ units receive one of the treatments through a randomized assignment mechanism, and the observed outcomes from the two treatment groups are compared to draw inference on causal estimands of interest. Such causal estimands can be defined in terms of potential outcomes, and the framework for drawing inference in the setup described above is the well-known Neyman-Rubin causal model \citep{sekhon2008neyman} or simply the Rubin causal model \citep{holland1986statistics}. In particular, this model posits that each unit $i=1,\dots,N$ has two fixed potential outcomes, $Y_i(1)$ and $Y_i(0)$, denoting unit $i$'s outcome under treatment and control, respectively, and thus the only stochasticity in a randomized experiment is units' assignment to treatment groups.

The most common estimand in the causal inference literature is the finite-population average treatment effect (ATE), defined as $\tau = \frac{1}{N} \sum_{i=1}^N [Y_i(1) - Y_i(0)]$. In a randomized experiment, ultimately each unit can only be assigned to treatment or control---never both---and thus only $Y_i(1)$ or $Y_i(0)$ is observed for each unit. As a result, $\tau$ must be estimated. The most common estimator for $\tau$ is the mean-difference estimator, i.e., the mean difference in the observed treatment outcomes and observed control outcomes. In this finite population setup where the potential outcomes are assumed fixed, the mean-difference estimator is unbiased for $\tau$, and its sampling variance was derived by \cite{splawa1990application}. Using this estimator and sampling variance, a normal approximation can be used to draw inference on $\tau$. If additional covariates are available for each unit, covariate adjustment can be performed to obtain more precise estimators for $\tau$, such as through post-stratification \citep{miratrix2013adjusting} or regression \citep{lin2013agnostic}.

However, in many practical situations, it may not be possible to observe a response for all $N$ units due to resource constraints. A natural way to address this limitation is to first randomly sample $n$ units from the population and then randomly assign the two treatments to the sampled units. For example, this is a common procedure in political science and other social sciences, where an experiment is conducted on a representative sample of some larger population, often through a well-designed online survey \citep{mutz2011population,coppock2018generalizing,miratrix2018worth}. In this case, the mean-difference estimator is still unbiased for the ATE, and its sampling variance has been studied in works such as \cite{imbens2004nonparametric} and \citet[Chapter 6]{imbens2015causal}. Then, one can still use the mean-difference estimator and its sampling variance to draw inference on $\tau$ via a normal approximation. This scenario---where $n$ units are sampled from a population of $N$ units, and then an experiment is conducted on these $n$ sampled units---is shown in the left panel of Figure \ref{Fig:Two_designs}. 

\begin{figure}
\centering
\begin{subfigure}[t]{0.51\textwidth}
\centering
	\includegraphics[scale=0.35]{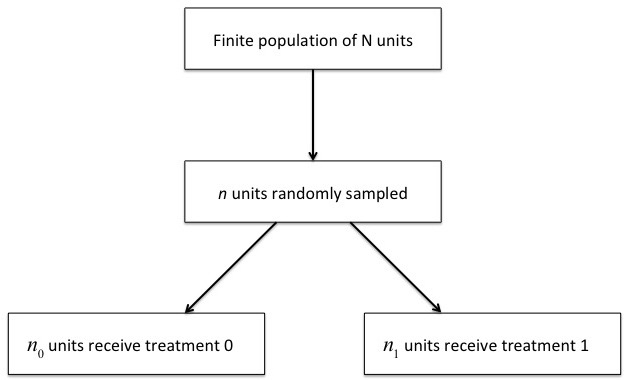}
	\caption{Design 1: Sampling in first stage and randomization in second stage.}
\end{subfigure}%
~
\begin{subfigure}[t]{0.51\textwidth}
\centering
	\includegraphics[scale=0.35]{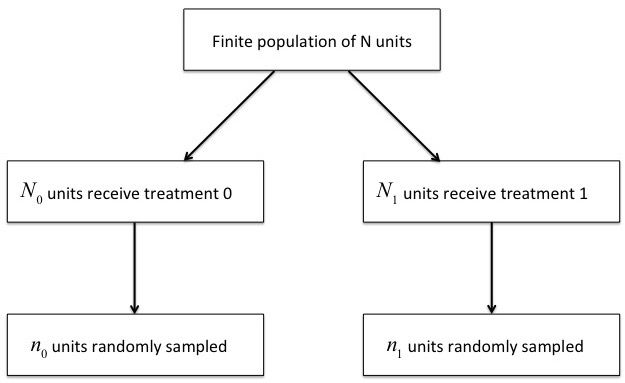}
	\caption{Design 2: Randomization in first stage and sampling in second stage.}
\end{subfigure}
\caption{Two sampling-based designs for estimation of ATE.} \label{Fig:Two_designs}
\end{figure}

This article is motivated by a similar limitation as stated in the previous paragraph, but the experimental design is different in the sense that the order of sampling and randomization is reversed. That is, first, each of the $N$ units is assigned to one of the two treatments. Then, a subset of units is sampled from the units exposed to each treatment group, and the response is measured for each sampled unit. This experimental design scenario is shown in the right panel of Figure \ref{Fig:Two_designs}. Such a strategy may be useful if sampling and measurement of response is more expensive and/or complex compared to treatment assignment.  Examples of experiments that involve this type of design can be found in material science, biomedicine, and the social sciences. For example, in an experiment conducted to assess the difference between two types of oxidation procedures with respect to their impact on the dimension of nanotubes, scientists typically split a population (a container of nanotubes) into two subpopulations (two smaller containers) and apply different oxidation procedures to the two subpopulations \citep{remillard2016electric,remillard2017tuning}. However, since measuring dimensions of nanotubes is an expensive and time-consuming process, a sample of oxidized nanotubes is taken from each subpopulation, and dimensions of the sampled nanotubes are measured. A similar procedure is followed when conducting stem cell experiments \citep{chung2005homeodomain,doi2009differential}.

The main contribution of this article is deriving the sampling properties of the mean-difference estimator for the ATE under Design 2---we do this in Section \ref{s:properties} after setting up some notation in Section \ref{s:notations}. We will find that these sampling properties are identical to those under Design 1, meaning that the ordering of the sampling and randomization stages is inconsequential for the precision of the mean-difference estimator. This finding will also bring some clarifications about inference of the ATE under Design 1, as we discuss in Section \ref{ss:comparison}. In particular, the variance of the mean-difference estimator under Design 1 is often characterized by asymptotic results where $N >> n$ or $N = \infty$ \citep{imbens2004nonparametric,imbens2015causal,sekhon2017inference}, and our work illuminates on the case where the population size and sample size are finite. Our work also gives insight into other sampling-based randomized designs: For example, Design 2 is similar---but not identical---to a cluster-randomized design \citep{campbell2007developments}, where one cluster of size $N_1$ is assigned to treatment group 1 and another cluster of size $N_0$ is assigned to treatment group 0, and random samples are obtained within each cluster. Cluster-randomized designs are quite common in education \citep{hedges2007intraclass}, medicine \citep{eldridge2004lessons}, and psychology \citep{raudenbush1997statistical}, and we discuss how Design 2 compares to cluster-randomized designs in Section \ref{ss:clusterRandomizedDesigns}. Then, in Section \ref{s:pretreat}, we consider the case where baseline (pre-treatment) measurements of the response can be obtained for a sample of experimental units, and we explore the extent to which the mean-difference estimator can be improved under Design 2 by using such measurements. Unfortunately, we will find that pre-treatment measurements are often unhelpful in improving the precision of average treatment effect estimators under this design, unless a large number of pre-treatment measurements that are highly associative with the post-treatment measurements can be obtained. In Section \ref{s:simulations}, we confirm these results by conducting a simulation study based on a real experiment in nanomaterials. In Section \ref{s:discussion}, we conclude.

\section{Setup and Notations} \label{s:notations}

Let the two treatments be denoted by 0 and 1, and for $i=1, \ldots, N$, let $(Y_i(0), Y_i(1))$ denote the potential outcomes for unit $i$ when exposed to treatments 0 and 1, respectively. The unit-level treatment effect is
$$ \tau_i = Y_i(1) - Y_i(0), \quad i=1, \ldots, N, $$
and the finite-population ATE is $$ \tau = N^{-1} \sum_{i=1}^N \tau_i = \bar{Y}(1) - \bar{Y}(0), $$
where $\bar{Y}(1) = N^{-1} \sum_{i=1}^N Y_i(1)$ and $\bar{Y}(0) = N^{-1} \sum_{i=1}^N Y_i(0)$ denote the average potential outcomes for the treatment groups 1 and 0, respectively. We also denote the variances of the potential outcomes for treatment group $T \in \{0,1\}$ as follows:
\begin{eqnarray}
S^2_{T} &=& (N-1)^{-1} \sum_{i=1}^N \left(Y_i(T) - \bar{Y}(T) \right)^2,  \label{eq:varT}
\end{eqnarray}
and the covariance of potential outcomes for treatment groups 1 and 0 by
\begin{equation}
S_{10} = (N-1)^{-1} \sum_{i=1}^N \left(Y_i(1)-\bar{Y}(1) \right) \left(Y_i(0)-\bar{Y}(0) \right). \label{eq:cov10}
\end{equation}
Finally, the variance of the unit-level treatment effects is
\begin{equation}
S^2_{\tau} = (N-1)^{-1} \sum_{i=1}^N \left(\tau_i - \tau \right)^2.  \label{eq:vartau}
\end{equation}
 
As in Design 2 shown in the right panel of Figure \ref{Fig:Two_designs}, $N_{1}$ and $N_{0}$ units are assigned to treatments 1 and 0 respectively, where $N_{1}$ and $N_{0}$ are predetermined. From these two groups, $n_{1}$ and $n_{0}$ units are sampled and their responses are observed. Let $\bar{y}(1)$ and $\bar{y}(0)$ denote the observed averages for the treatment groups 1 and 0, respectively. Then, a natural estimator of $\tau$ is
\begin{equation}
\hat{\tau} = \bar{y}(1) - \bar{y}(0). \label{eq:tauhat}
\end{equation}
We will examine the sampling properties of $\hat{\tau}$ and compare them with those of a similar estimator $\tilde{\tau}$ obtained from Design 1 shown in Figure \ref{Fig:Two_designs}.
 
\section{Sampling properties of the estimator of the ATE} \label{s:properties}

Because the potential outcomes are assumed fixed, the sampling properties of $\hat{\tau}$ defined in (\ref{eq:tauhat}) will be determined by the random variables associated with the randomization and sampling stages. We define two such variables now. For $i=1, \ldots, N$, let $T_i$ denote a Bernoulli random variable indicating the random level of treatment (0 or 1) that unit $i$ receives. Recalling that the assignment mechanism essentially involves a random partitioning of $N$ units into two groups of predetermined sizes $N_0$ and $N_1$, the properties of the assignment vector $(T_1, \ldots, T_N)$ are straightforward to establish and can be found in standard texts (e.g., \citealt{imbens2015causal}).

Next, define $S_{i}^{T_i}$ as an indicator random variable equaling $1$ if the potential outcome $Y_{i}(T_i)$ is randomly sampled among the units assigned to treatment level $T_i \in \{0, 1\}$. Importantly, note that we know that $S_{i}^{1} = 0$ conditional on $T_i = 0$; i.e., the $i$th unit will not be sampled from treatment group 1 if it was assigned to treatment group 0. As shown in the right panel of Figure \ref{Fig:Two_designs}, we assume that samples of size $n_0$ and $n_1$ are sampled from treatment groups 0 and 1, respectively. Properties of the $S_{i}^{T_i}$ are crucial in the derivation of sampling properties of $\hat{\tau}$, and are summarized in Lemma \ref{lem:S_properties}. The proofs are in Appendix \ref{appendixA}.

\begin{lemma} \label{lem:S_properties}
The properties of the sampling indicators $S_i^{T_i}$ can be summarized as:
\begin{eqnarray}
	\E \left[S_{i}^{T_i} \right] &=& n_{T_i}/N,  \quad i = 1, \ldots, N,  \label{eq:S1} \\
	\var \left[ S_{i}^{T_i} \right] &=& n_{T_i}/N \{ 1 - n_{T_i}/N \}, \quad i = 1,\ldots, N, \label{eq:S2} \\
	\cov \left[ S_i^{T_i}, S_{i^{\prime}}^{T_{i^{\prime}}} \right] &=&  \left \{ \begin{array}{ll}
	- n_{T_i} n_{1-T_i} /N^2, & \mbox{if} \quad i=i^{\prime}, T_i \ne T_{i^{\prime}} \\
  - \frac{n_{T_i} \left\{ N - n_{T_i} \right\}}{N^2(N-1)}, & \mbox{if} \quad i \ne i^{\prime}, T_i = T_{i^{\prime}}, \label{eq:S4} \\
\frac{n_{T_i} n_{1-T_i}}{N^2(N-1)}, & \mbox{if} \quad i \ne i^{\prime}, T_i \ne T_{i^{\prime}} 
\end{array}
\right.
	\end{eqnarray}
\end{lemma}

\begin{remark} \label{rem:lemma2}
To provide some intuition for (\ref{eq:S4}), note that we expect the quantity $\cov \left[ S_i^{T_i}, S_{i^{\prime}}^{T_{i^{\prime}}} \right]$ to be negative if $T_i = T_{i^{\prime}}$, because if we know that unit $i^{\prime}$ is sampled in treatment group $T_i$, then unit $i$ is less likely to be sampled in the same group. Similarly, we expect the quantity $\cov \left[ S_i^{0}, S_{i^{\prime}}^{1} \right]$ to be positive, because if we know that unit $i^{\prime}$ is sampled in treatment group 1, then unit $i$ is more likely to be sampled in treatment group 0.
\end{remark}

We now represent the estimator $\hat{\tau}$ in (\ref{eq:tauhat}) in terms of the potential outcomes and the random variables $T_i$ and $S_i^{T_i}$ defined above. To do this, note that the average observed response in treatment group $T \in \{0, 1\}$ can be written as
\begin{equation}
\bar{y}(T) = (n_{T})^{-1} \sum_{i=1}^N S_i^{T} \ Y_i(T), \ i=1, \ldots, N. \label{eq:groupavg}
\end{equation}
By combining (\ref{eq:tauhat}) and (\ref{eq:groupavg}) and using Lemma \ref{lem:S_properties}, we now derive the sampling properties of the estimator $\hat{\tau}$ and summarize them in Theorem \ref{thm:tauhat_properties}. The proof of Theorem \ref{thm:tauhat_properties} is in Appendix \ref{appendixB}.

\begin{theorem} \label{thm:tauhat_properties}
The estimator $\hat{\tau}$ given by (\ref{eq:tauhat}) satisfies the following properties:
\begin{enumerate}
\item $\hat{\tau}$ is an unbiased estimator of the ATE $\tau$.
\item The sampling variance of $\hat{\tau}$ is given by
\begin{equation}
\var(\hat{\tau}) = S_1^2/n_1 + S_0^2/n_0 - S^2_{\tau}/N, \label{eq:vartauhat}
\end{equation}
where $S^2_0$ and $S^2_1$ are given by (\ref{eq:varT}) and $S^2_{\tau}$ is given by (\ref{eq:vartau}).
\end{enumerate}
\end{theorem}

\subsection{Comparison with two other sampling-based randomized designs} \label{ss:comparison}

We now compare the sampling properties of $\hat{\tau}$ as derived in Theorem \ref{thm:tauhat_properties} with the unbiased estimators of $\tau$ obtained from two other designs: (i) a design where responses for all units are observed, that is, no sampling is involved, and (ii) a design described as Design 1 in the left panel of Figure \ref{Fig:Two_designs}, where $n$ units are first sampled from $N$ units in the population and subsequently exposed to treatments. We denote the estimator of $\tau$ from (i) by $\widehat{\widehat{\tau}}$ and that from (ii) by $\widetilde{\tau}$. Both $\widehat{\widehat{\tau}}$ and $\widetilde{\tau}$ are unbiased estimators of $\tau$. \cite{splawa1990application} derived the following result on the sampling variance of $\widehat{\widehat{\tau}}$:
\begin{equation}
\var \left( \widehat{\widehat{\tau}} \right) = S_1^2/N_1 + S_0^2/N_0 - S^2_{\tau}/N. \label{eq:varNeyman}
\end{equation}
As a result, we obtain the inequality $\var \left( \widehat{\widehat{\tau}} \right) < \var \left(\hat{\tau} \right)$, reflecting the price one has to pay for sampling from the finite population.

Meanwhile, to our knowledge, \citet[Chapter 6, Appendix B]{imbens2015causal} were the first to derive the expression of the finite-population sampling variance of $\widetilde{\tau}$, although others have discussed the infinite-population case where $N = \infty$ (e.g., \citealt{imbens2004nonparametric}). However, there was a slight error in the derivation in \citet[Chapter 6, Appendix B]{imbens2015causal}: The covariance between the sampling indicators was stated to be $-\left(\frac{n}{N}\right)^2$, when really it is $-\frac{n(N-n)}{N^2(N-1)}$, as can be seen by the $i\neq i'$, $T_i = T_{i'}$ case in Lemma \ref{lem:S_properties} and other works on survey sampling (e.g., \citealt[Page 53]{lohr2009sampling}). As we show in Appendix \ref{appendixC}, by making this correction, and then following the derivation of \citet[Chapter 6, Appendix B]{imbens2015causal}, we have the following expression of the sampling variance of $\widetilde{\tau}$:
\begin{align}
	\var \left(\widetilde{\tau} \right) = S_1^2/n_1 + S_0^2/n_0 - S^2_{\tau}/N = \var \left(\hat{\tau} \right) \label{eqn:comparison}
\end{align}
Thus, the precision of the mean-difference estimator under Design 1 is identical to that under Design 2; in other words, the order of the sampling and randomization stages is inconsequential in terms of the precision of the mean-difference estimator. This finding also clarifies some discussions in the causal inference literature that consider a super-population framework. When discussing causal inference under a super-population framework, many works implicitly assume that $N >> n$ or $N = \infty$, and thus the third term in $\var \left(\widetilde{\tau} \right)$, $S^2_{\tau}/N$, is often ignored (\cite{imbens2004nonparametric}, \citet[Chapter 6]{imbens2015causal}, \cite{ding2017bridging}, \cite{sekhon2017inference}). As we will see in Section \ref{ss:var_estimation}, this term is ignored in the \textit{estimation} of $\var (\hat{\tau})$ or $\var (\widetilde{\tau})$, because the observed data do not provide any information about $S^2_{\tau}$, since none of the individual $\tau_i$ are ever observed. However, the above derivation emphasizes that this term nonetheless exists in the \textit{true} variance of the mean-difference estimator under Design 1 and Design 2 when the super-population that is sampled from is finite.

\subsection{Comparison with cluster-randomized designs} \label{ss:clusterRandomizedDesigns}

As discussed in Section \ref{s:intro}, Design 2 is similar but not identical to a cluster-randomized design, where treatment is assigned among clusters of units but response is measured at the unit level. The most important distinction between Design 2 and a cluster-randomized design is that, in Design 2, $N_0$ and $N_1$ (the number of units assigned to treatment groups 0 and 1, respectively) are fixed, and in a cluster-randomized design, these quantities are stochastic. For example, consider a cluster-randomized design with two clusters, where Cluster A has $N_A$ units and Cluster B has $N_B$ units, and one cluster is randomly assigned to treatment group 0 and the other to treatment group 1. There is a 0.5 probability that $N_0 = N_A$ and a 0.5 probability that $N_0 = N_B$, and $N_1$ is analogously stochastic.

This stochasticity causes many complications to inference---for example, the mean-difference estimator is often biased in cluster-randomized designs, even when all units' responses can be measured \citep{middleton2008bias,middleton2015unbiased}. Consequently, there is not a straightforward analog to our results that will hold for cluster-randomized designs; indeed, our results will not even hold for the simple two-cluster example mentioned in the previous paragraph. Because we require that $N_0$ and $N_1$ are fixed, our results will hold for cluster-randomized designs if the number of units within each cluster is the same across clusters, but this is rarely the case. As discussed in \cite{middleton2008bias} and \cite{middleton2015unbiased}, cluster sizes and the covariance between treatment group sizes and treatment effects are important quantities for deriving inferential properties of ATE estimators in cluster-randomized designs, and likely these quantities would be similarly important in deriving analogous results when there is a sampling stage after treatment assignment at the cluster level. We leave this for future work.

\subsection{Estimation of sampling variance and approximate confidence intervals} \label{ss:var_estimation} 

Let
\begin{equation}
s^2_0 = \frac{1}{n_0 - 1} \sum_{i: S_i^{0}=1} \left\{ Y_i(0) - \bar{y}(0) \right\}^2 \quad \mbox{and} \quad
s^2_1 = \frac{1}{n_1 - 1} \sum_{i: S_i^{1}=1} \left\{ Y_i(1) - \bar{y}(1) \right\}^2  \label{eq:sample_var}
\end{equation}
denote the sample variances of observed responses for treatment groups 0 and 1, respectively. From standard sampling theory (e.g., \citealt{lohr2009sampling}, Pages 52-54), it follows that $s^2_0$ and $s^2_1$ are unbiased estimators of $S^2_0$ and $S^2_1$, respectively. Consequently,
\begin{equation}
v_{\hat{\tau}} = s^2_0/n_0 + s^2_1 /n_1 \label{eq:var_estimator}
\end{equation}
is a natural Neymanian-style estimator of $\var(\hat{\tau})$ and $\var(\widetilde{\tau})$, which, as seen in Section \ref{ss:comparison}, has an upward bias of $S^2_{\tau}/N$ unless strict additivity holds. This estimator---originating in \cite{splawa1990application}---is by far the most common estimator for the variance of the ATE in randomized experiments \citep{rubin1990comment,imbens2004nonparametric,miratrix2013adjusting,imbens2015causal,ding2017bridging}, and thus it is reassuring that it can be used for $\var(\hat{\tau})$ and $\var(\widetilde{\tau})$, i.e., it can be used under Design 1 or Design 2.

Estimator (\ref{eq:var_estimator}) can be used to obtain approximate confidence intervals for $\tau$ as 
$$ \hat{\tau} \pm \Phi^{-1}(1-\alpha/2) \sqrt{v_{\hat{\tau}}}, $$
where $\Phi^{-1}(\cdot)$ denotes the quantile function of a standard normal distribution. The asymptotic normality of $\hat{\tau}$ is based on the finite population central limit theorem \citep{hajek1960limiting} and its application in the context of randomized experiments by \cite{li2017general}. 

\section{Can the estimator be improved using pre-treatment measurements?} \label{s:pretreat}

Now assume that the response can be measured for each experimental unit prior to application of the treatments. Let $X_1, \ldots, X_N $ denote these measurements for the $N$ units, which are fixed quantities like the potential outcomes. The unit-level differences $\Delta_{1i} = Y_i(1) - X_i$ and $\Delta_{0i} = Y_i(0) - X_i$ for $i=1, \ldots, N$ are referred to as ``gain scores''  in the psychology \citep{bellini2017research} and education \citep{mcgowen2002growth} literature. While the unit-level gain scores or their averages $\Delta_1 = N^{-1} \sum_{i=1}^N \Delta_{1i}$ and $\Delta_0 = N^{-1} \sum_{i=1}^N \Delta_{0i}$ are of interest in psychology and education \citep{hake1998interactive}, they are not referred to as \textit{causal} estimands or \textit{causal} effects \citep{rubin2004potential,imbens2015causal}. However, it is easy to see that the unit-level and average treatment effects can respectively be expressed as the unit-level and average differences of the gain scores, i.e., $\tau_i = \Delta_{1i} - \Delta_{0i}$ and $\tau = \Delta_1 - \Delta_0$. In spite of this connection, several experts in education and psychology have recommended avoiding the use of gain scores when estimating treatment effects in experiments and observational studies. \cite{campbell1970regression} claimed, ``gain scores are in general $\ldots$ a treacherous quicksand,'' and \cite{cronbach1970we} recommended researchers to ``frame their questions in other ways.'' Despite some recent interest on utilizing gain scores to identify causal effects, there appears to be a general aversion in the causal inference community towards the use of gain scores. Here we explore whether, in the current setup, a proper design and analysis of the experiment using gain scores can potentially lead to more precise estimation of treatment effects under certain assumptions.
 
In the setup of Design 2 in Figure \ref{Fig:Two_designs}, where measuring the pre-treatment response for each unit is not feasible, it is often a common practice to measure the pre-treatment response for a random sample of size $n_x$ to estimate the average gain scores $\Delta_1$ and $\Delta_0$. Denoting the sampling indicator by $Z_1, \ldots, Z_N$, the estimators of $\Delta_1$ and $\Delta_0$ are $\bar{y}(1) - \bar{x}$ and $\bar{y}(0) - \bar{x}$, respectively, where $\bar{x} = n_x^{-1} \sum_{i=1}^N Z_i X_i$. While the sampling properties of these estimators can be readily obtained, they do not help in increasing the precision of the estimator of $\hat{\tau}$ because the average pre-treatment scores get canceled out in the difference of the gain score estimators. However, if samples of $n_{x_1}$ and $n_{x_0}$ pre-treatment observations are taken independently from the treatment groups 1 and 0 \textit{after assignment} but \textit{before administration} of the treatments, it is possible to obtain a different estimator of the ATE. For $i = 1, \ldots, N$, let $Z_i^{T_i}$ denote the sampling indicator associated with the random sampling of $X_i$ among the units assigned to treatment $T_i$. Then 
\begin{equation}
\bar{x}(1) = (n_{x_1})^{-1} \sum_{i=1}^N Z_i^1 X_i, \quad \mbox{and} \quad \bar{x}(0) = (n_{x_0})^{-1} \sum_{i=1}^N Z_i^0 X_i, \label{eq:x_avg}
\end{equation}
are the observed sample averages of pre-treatment responses for the two treatment groups. Then we can define the following estimator of the ATE:
\begin{equation}
\hat{\tau}^* = \hat{\Delta}_1 - \hat{\Delta}_0 = \left\{ \bar{y}(1) - \bar{x}(1) \right\} - \left\{ \bar{y}(0) - \bar{x}(0) \right\} =\hat{\tau} - \left\{ \bar{x}(1) - \bar{x}(0) \right\}, \label{eq:improved_est} 
\end{equation} 
where $\bar{x}(1)$ and $\bar{x}(0)$ are given by (\ref{eq:x_avg}). The sampling properties of $\hat{\tau}^*$ depend on the distribution of $Z^{T_i}$ and the joint distribution of $S^{T_i}$ and $Z^{T_i}$. These properties are summarized in the following two lemmas.

\begin{lemma} \label{lem:Z_properties}
The properties of the indicators $Z_i^{T_i}$ are exactly the same as those of $S_i^{T_i}$ stated in Lemma \ref{lem:S_properties}, just replacing $n_{T_i}$ by $n_{x_{T_i}}$ and $n_{1-T_i}$ by $n_{x_{1 - T_i}}$.
\end{lemma}

\begin{lemma} \label{lem:SZ_properties}
For $i, i^{\prime} = 1, \ldots, N$, the covariance between the indicators $S_i^{T_i}$ and $Z_{i^{\prime}}^{T_{i^{\prime}}}$ is given by:
\begin{eqnarray*}
\cov\left[ S_i^{T_i}, Z_{i^{\prime}}^{T_{i^{\prime}}} \right] = 
\left \{ \begin{array}{ll}
\frac{n_{T_i} n_{x_{T_i}} N_{1-T_i}}{N_{T_i} N^2}, & \mbox{if} \quad i=i^{\prime}, T_i = T_{i^{\prime}}, \\
- \frac{n_{T_i} n_{x_{1-T_i}}}{N^2}, & \mbox{if} \quad i=i^{\prime}, T_i \ne T_{i^{\prime}}, \\
- \frac{n_{T_i} n_{x_{T_i}} N_{1-T_i}}{N_{T_i} N^2 (N-1) } & \mbox{if} \quad i \ne i^{\prime}, T_i = T_{i^{\prime}}, \\
\frac{n_{T_i} n_{x_{1-T_i}}}{N^2(N-1)}& \mbox{if} \quad i \ne i^{\prime}, T_i \ne T_{i^{\prime}}.
\end{array} \right.
\end{eqnarray*}
\end{lemma}
\noindent
The proof for Lemma \ref{lem:SZ_properties} is in Appendix \ref{appendixD}. Using Lemmas \ref{lem:Z_properties} and \ref{lem:SZ_properties}, we arrive at the following result:
\begin{theorem} \label{thm:tauhatstar_properties}
The estimator $\hat{\tau}^*$ given by (\ref{eq:improved_est}) satisfies the following properties:
\begin{enumerate}
\item $\hat{\tau}^*$ is an unbiased estimator of the ATE $\tau$.
\item The sampling variance of $\hat{\tau}^*$ is given by
\begin{equation}
\var(\hat{\tau}^*) = \var(\hat{\tau}) + \left(\frac{1}{n_{x_1}} + \frac{1}{n_{x_0}} \right) S_x^2 - 2 \left( \frac{S_{1x}}{N_1} + \frac{S_{0x}}{N_0} \right), \label{eq:vartauhatstar}
\end{equation}
where $\var(\hat{\tau})$ is given by (\ref{eq:vartauhat}) and
\begin{eqnarray}
S^2_x &=& (N-1)^{-1} \sum_{i=1}^N (X_i - \bar{X})^2, \label{eq:varx} \\ 
S_{Tx} &=& (N-1)^{-1} \sum_{i=1}^N \left\{ Y_i(T) - \bar{Y}(T) \right\} \{X_i - \bar{X}\}, \quad T \in \{0,1\}. \label{eq:cov_yx}
\end{eqnarray}
\end{enumerate}
\end{theorem}
\noindent
The proof of Theorem \ref{thm:tauhatstar_properties} is in Appendix \ref{appendixE}.

\begin{remark} \label{rem:thm3_implications}
From Theorem \ref{thm:tauhatstar_properties}, it follows that $\hat{\tau}^*$ is a more efficient estimator than $\hat{\tau}$ if and only if 
\begin{equation}
\begin{aligned}
2 \left( \frac{S_{1x}}{N_1} + \frac{S_{0x}}{N_0} \right) &> \left( \frac{1}{n_{x_1}} + \frac{1}{n_{x_0}} \right) S_x^2, \text{ which implies} \label{eq:eff_cond} \\
2 \left( \frac{S_1}{N_1} r_{1x} + \frac{S_0}{N_0} r_{0x} \right) &> \left( \frac{1}{n_{x_1}} + \frac{1}{n_{x_0}} \right) S_x,
\end{aligned}
\end{equation}
\noindent
where $r_{Tx} = \frac{S_{Tx}}{S_T S_x}$ denotes the finite sample correlation coefficient between the potential outcomes $Y(T)$ and pre-treatment measurements $X$. In order for condition (\ref{eq:eff_cond}) to be achieved, pre-treatment measurements need to be highly predictive of the outcome, post-treatment outcomes need to be substantially variable relative to the pre-treatment measurements, and pre-treatment measurements need to be obtained for a large portion of the population. For example, consider the case where we have a balanced design (i.e., $N_1 = N_0 = N/2$) and a balanced sample size ($n_{x_1} = n_{x_0} = n_x/2$). In this case, even if the pre-treatment measurements are perfectly correlated with the outcomes (i.e., $r_{1x} = r_{0x} = 1$), in order for (\ref{eq:eff_cond}) to be achieved, we need $\frac{n_x}{N} (S_1 + S_0) > S_x$, and usually $\frac{n_x}{N}$ will be small due to resource constraints. Similarly, if the pre-treatment measurements are moderately correlated with the outcomes (e.g., $r_{1x} = r_{0x} = 0.5$) and all units' pre-treatment measurements are observed (i.e., $n_x = N$), we still need $\frac{S_1 + S_0}{2} > S_x$, i.e., the average standard deviation of the post-treatment outcomes needs to be larger than the standard deviation of the pre-treatment measurements. Cases where $r_{1x} = r_{0x} = 1$ or $n_x = N$ are indeed quite extreme and unrealistic conditions, and thus Theorem \ref{thm:tauhatstar_properties} in general gives credence to the skepticism many causal inference experts have about using gain scores for estimation of treatment effects in the context of the experimental design discussed here. 
\end{remark}

\section{Simulation Study: Considering an Experiment in Nanomaterials} \label{s:simulations}

To assess how the findings from Theorems \ref{thm:tauhat_properties} and \ref{thm:tauhatstar_properties} are informative for real experiments, we now conduct a simulation study based on a real experiment in nanomaterials \citep{remillard2017tuning}. One purpose of this experiment was to assess how the dimensions of carbon nanotubes change as they undergo different processing procedures. One such procedure is sonication, the primary purpose of which is to disperse the nanotubes. However, sonication may also result in the undesirable outcome of breaking the nanotubes, possibly causing changes to their length. The experiment considered two types of sonication---bath sonication and probe sonication---which are the treatments in this application. Bath sonication is a more gentle procedure than probe sonication, so it was hypothesized that bath sonication would not decrease the length of the carbon nanotubes as much as probe sonication.

Because of practical constraints, this experiment was conducted using Design 2 in Figure \ref{Fig:Two_designs}. Because the nanotubes are so small, it was infeasible to select individual nanotubes for treatment, and instead, a container of nanotubes was evenly divided into two smaller containers, and each of these containers underwent bath sonication or probe sonication. After sonication, nanotubes were randomly selected and their length was measured after treatment. Furthermore, the length of a random sample of nanotubes was measured before treatment. Primary interest was in the average difference in log-length between bath sonication and probe sonication, i.e., the ATE; the log-length was used due to skewness in the length distribution across carbon nanotubes.  

In this paper, we have discussed two ATE estimators under Design 2: the mean-difference estimator $\hat{\tau}$ defined in (\ref{eq:tauhat}), and the mean-difference estimator $\hat{\tau}^*$ defined in (\ref{eq:improved_est}) that incorporates pre-treatment information. In the real experiment conducted in \cite{remillard2017tuning}, we could only implement Design 2 once, and thus we could not observe the behavior of these estimators across different randomizations and random samples of carbon nanotubes. In order to understand the behavior of these estimators across many implementations of Design 2, we will conduct a simulation study mimicking this experiment.\footnote{In actuality, the \cite{remillard2017tuning} experiment considered several treatment factors, such as oxidation as well as sonication, and it considered several types of carbon nanotubes. For ease of exposition, the simulation data discussed here is based on the subset of the experiment that considered the carbon nanotube called ``D15L1-5'' in \cite{remillard2017tuning}. We chose this type of carbon nanotube because it was already oxidated, and thus only sonication (i.e., one treatment factor) was used in the actual experiment for this type of carbon nanotube.}

Consider the following implementation of Design 2 for this experiment: $N = 1,000$ carbon nanotubes will be randomly divided into treatment groups of sizes $N_1 = 500$ (probe sonication) and $N_0 = 500$ (bath sonication). Random samples of sizes $n_{x1} = 150$ and $n_{x0} = 150$ will be obtained before treatment, and random samples of sizes $n_1 = 150$ and $n_0 = 150$ will be obtained after treatment; these were approximately the sample sizes that were used in the actual experiment. The individual log-length of each carbon nanotube will be measured for these samples.

First, the population of pre-treatment measurements $\mathbf{X}$ as well as post-treatment measurements $\mathbf{Y}(1)$ and $\mathbf{Y}(0)$ were generated using the following model:
\begin{equation}
\begin{aligned}
\label{eqn:simulationModel}
	\begin{pmatrix}
		\mathbf{X} \\
		\mathbf{Y}(0)
	\end{pmatrix} &\sim \mathcal{N} \begin{pmatrix}
		\begin{pmatrix}
		\mu_X \\
		\mu_{Y0}
	\end{pmatrix}, \begin{pmatrix}
		\sigma^2_X & \rho \sigma_X \sigma_{Y0} \\
		\rho \sigma_X \sigma_{Y0} & \sigma^2_{Y0}
	\end{pmatrix}
	\end{pmatrix} \\
	\mathbf{Y}(1) &= \mathbf{Y}(0) + \tau + \sigma_{\tau} \mathbf{Y}(0)
\end{aligned}
\end{equation}
As discussed in \cite{remillard2017tuning}, normality was a reasonable distributional assumption to place on the log-length of carbon nanotubes. The mean parameters ($\mu_X$ and $\mu_{Y0}$), variance parameters ($\sigma^2_X$ and $\sigma^2_{Y0}$), and ATE parameter $\tau$ were set to the sample estimates observed in the \cite{remillard2017tuning} experiment.\footnote{Specifically, $\mu_X = -0.09$, $\mu_{Y0} = 0.09$, $\sigma^2_X = 0.32$, $\sigma^2_{Y0} = 0.51$, and $\tau = -0.01$, where the units are on the $\log \mu m$ scale.} The treatment effect heterogeneity parameter $\sigma_{\tau}$ and association parameter $\rho$ could not be estimated in the experiment, because none of the $(\mathbf{X}, \mathbf{Y}(1), \mathbf{Y}(0))$ could ever be jointly observed. In this simulation study, we set $\sigma_{\tau} = 0.5$ to induce strong treatment effect heterogeneity \citep{ding2016randomization} and consider various values for $\rho$.

After the population-level triplet $\{\mathbf{X}, \mathbf{Y}(1), \mathbf{Y}(0)\}_{i=1,\dots,N}$ was generated \textit{once}, the following randomization and sampling procedures (i.e., Design 2) were repeated 10,000 times:
\begin{enumerate}
 	\item Each unit was randomly assigned to treatment group 1 or treatment group 0 with equal probability, such that $N_1 = 500$ and $N_0 = 500$.
 	\item Within each treatment group $T \in \{0,1\}$, $n_{xT} = 150$ units were randomly sampled and their $\mathbf{X}$ recorded as an observed pre-treatment measurement.
 	\item Within each treatment group $T \in \{0,1\}$, $n_{T} = 150$ units were randomly sampled and their $\mathbf{Y}(T)$ recorded as an observed post-treatment measurement.
 \end{enumerate} 
 For each of the 10,000 replications, we recorded the ATE estimators $\hat{\tau}$ defined in (\ref{eq:tauhat}) and $\hat{\tau}^*$ defined in (\ref{eq:improved_est}). Recall that $\hat{\tau}$ is the mean-difference estimator, and its properties under this design are established by Theorem \ref{thm:tauhat_properties}; meanwhile, $\hat{\tau}^*$ uses the pre-treatment measurements $\mathbf{X}$ to alter $\hat{\tau}$, and its properties are established by Theorem \ref{thm:tauhatstar_properties}.

First, let us consider one simulation setting to confirm the results in Theorems \ref{thm:tauhat_properties} and \ref{thm:tauhatstar_properties}. Figure \ref{fig:histogramComparison} shows the empirical distribution of $\hat{\tau}$ and $\hat{\tau}^*$ after 10,000 replications of the above procedure when the pre- and post-treatment measurements are moderately associated $(\rho = 0.5)$. There are two observations that can be made from Figure \ref{fig:histogramComparison}. First, the normal distributions for $\hat{\tau}$ and $\hat{\tau}^*$ (constructed using Theorems \ref{thm:tauhat_properties} and \ref{thm:tauhatstar_properties}, respectively) fit the empirical distributions of these estimators; this confirms our theoretical results, including the comparison of these estimators given in Remark \ref{rem:thm3_implications}. Second, under this setting, $\hat{\tau}^*$ is slightly more dispersed than $\hat{\tau}$, suggesting that using the pre-treatment measurements in this case inflates the variance of the ATE estimator.

\begin{figure}
	\centering
	\includegraphics[scale=0.5]{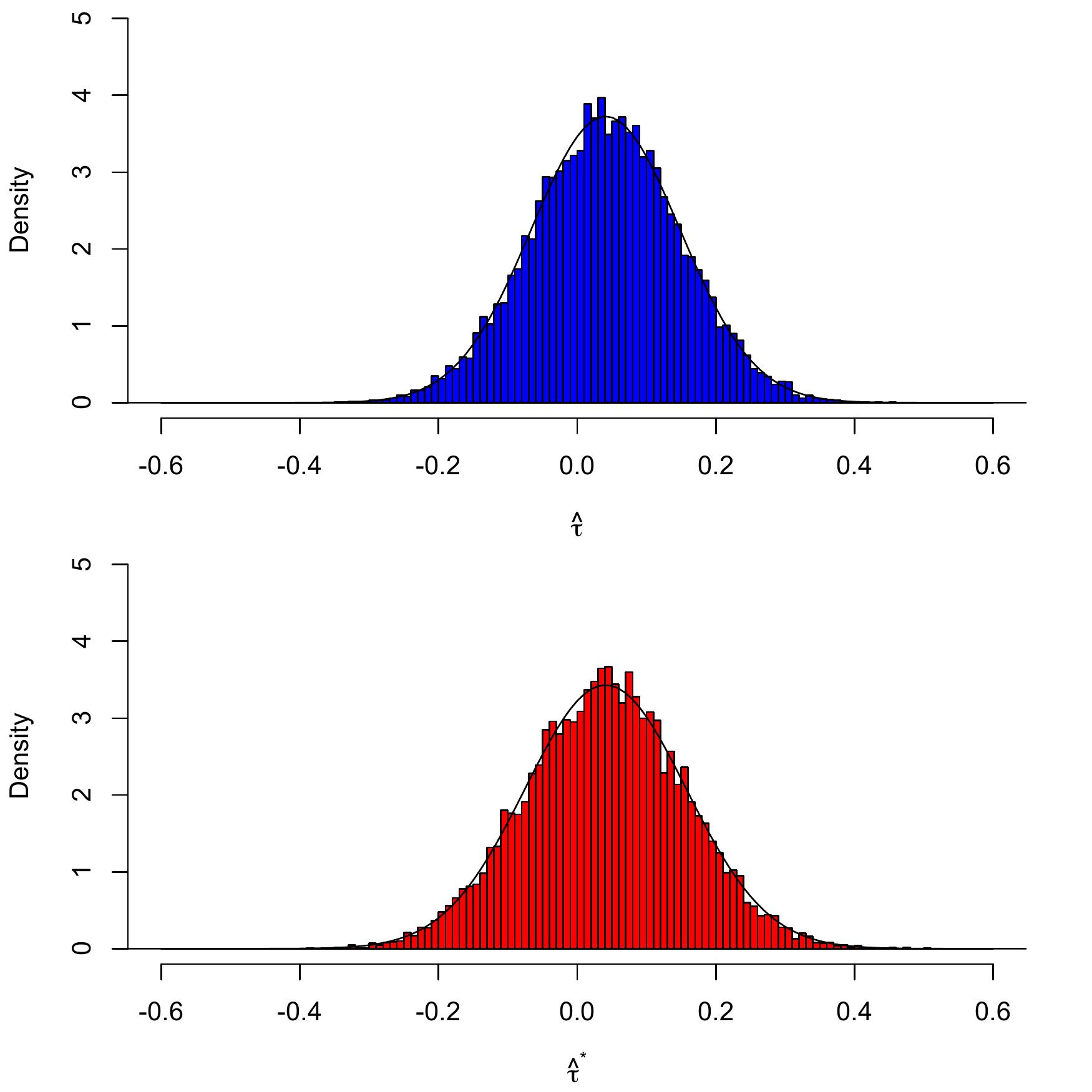}
	\caption{10,000 replications of $\hat{\tau}$ (top) and $\hat{\tau}^*$ (bottom) when the pre- and post-treatment measurements are moderately associated $(\rho = 0.5)$ and there is strong treatment effect heterogeneity $(\sigma_{\tau} = 0.5)$. The lines denote normal densities, where the mean is the true treatment effect and the variance is given by Theorem \ref{thm:tauhat_properties} for $\hat{\tau}$ and Theorem \ref{thm:tauhatstar_properties} for $\hat{\tau}^*$.}
	\label{fig:histogramComparison}
\end{figure}

This second observation begs the question: In this scenario, when will using the pre-treatment measurements improve estimation of the ATE? To address this question, we considered various values for the association parameter $\rho$ as well as the sample sizes $n_{xT}$ and $n_T$ for $T \in \{0,1\}$. Figure \ref{fig:varianceComparisonHeatmap} displays a heatmap of the empirical ratio $\frac{\text{Var}(\hat{\tau})}{\text{Var}(\hat{\tau}^*)}$ for various values of $\rho$ and the sample size. Looking at the bottom row of this heatmap, we can see that for the sample of size of 150 that was actually feasible for this experiment, even an association of $\rho = 0.9$ would not have led to $\hat{\tau}^*$ being more precise than $\hat{\tau}$. For this experiment, the sample size would have had to significantly increase while keeping the association moderately high in order for the pre-treatment measurements to be useful in increasing the precision of the ATE estimator. This echoes the observation made at the end of Section \ref{s:pretreat} that the use of gains scores is often not beneficial in improving ATE estimators.

\begin{figure}
	\centering
	\includegraphics[scale=0.5]{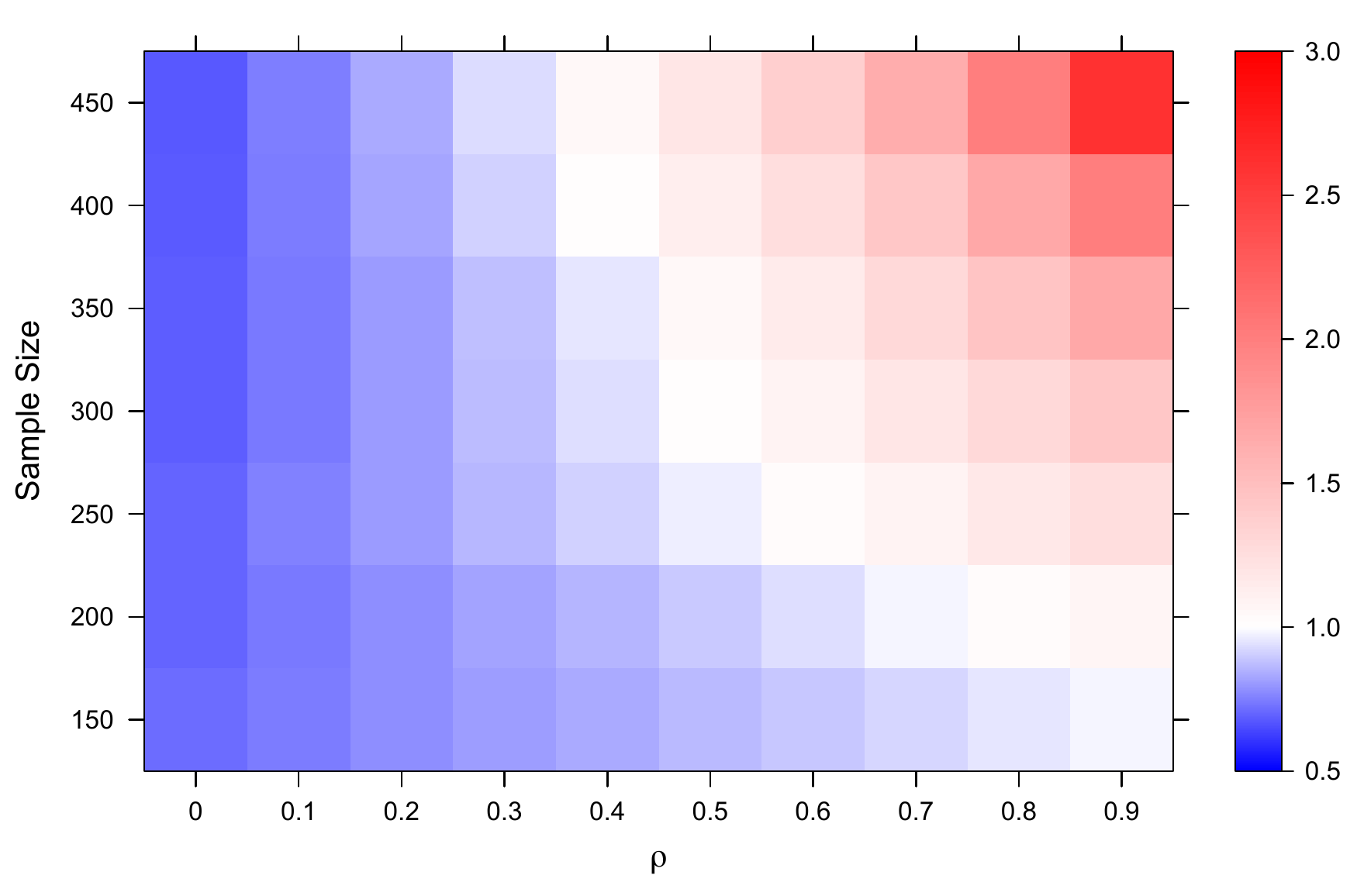}
	\caption{Heatmap of the empirical ratio $\frac{\text{Var}(\hat{\tau})}{\text{Var}(\hat{\tau}^*)}$ across the 10,000 replications. We considered values $\rho \in \{0, 0.1,\dots, 0.9\}$ in (\ref{eqn:simulationModel}) and sample sizes $150,200,\dots,450$ for $n_{x1}$, $n_{x0}$, $n_{1}$, and $n_{0}$. Blue shades suggest $\hat{\tau}$ is more precise; red shades suggest $\hat{\tau}^*$ is more precise.}
	\label{fig:varianceComparisonHeatmap}
\end{figure}

In summary, the above simulation study confirms that Theorems \ref{thm:tauhat_properties} and \ref{thm:tauhatstar_properties} correctly establish the behavior of the ATE estimators $\hat{\tau}$ and $\hat{\tau}^*$ under Design 2. However, for these data that are based on the experiment from \cite{remillard2017tuning}, an infeasibly large number of pre-treatment measurements would need to be obtained in order for $\hat{\tau}^*$ to be more precise than $\hat{\tau}$, and these measurements would need to be at least moderately associated with the post-treatment measurements. As noted in Remark \ref{rem:thm3_implications}, this will also be the case for other experiments, unless the post-treatment measurements are substantially more variable than the pre-treatment measurements, which was not the case for the \cite{remillard2017tuning} experiment.

\section{Discussion} \label{s:discussion}

Many causal inference works have focused on experimental settings where the outcomes for all units in the experiment can be measured. In some settings, it is too expensive to conduct an experiment on all units of interest, and instead an experiment is conducted on a random sample of units. Texts such as \cite{imbens2015causal} have shown that the inferential properties of common treatment effect estimators in these settings can be established by first accounting for the stochasticity of the sampling stage and then accounting for the stochasticity of the randomization stage. However, inferential properties under the experimental design scenario where the ordering of the sampling and randomization stages are reversed has not been established. Forms of this experimental design have become increasingly common in the physical, medical, and social sciences, and so it is important to understand statistical inference in this case.

We established the inferential properties of the mean-difference estimator under this experimental design scenario, and we compared our findings to results for other experimental designs. We found that the inferential properties of the mean-difference estimator under this experimental design scenario are identical to those under the sample-first-randomize-second design, which is the more common experimental design discussed in the literature. Thus, the ordering of the randomization and sampling stages is inconsequential for inference of average treatment effects. We also assessed if pre-treatment measurements of units' outcomes can be used to improve upon the mean-difference estimator for this experimental design scenario. We found that this is only the case if (1) the pre-treatment measurements are highly predictive of the outcome, (2) the post-treatment outcomes are substantially variable relative to the pre-treatment measurements, and (3) pre-treatment measurements are obtained for a large portion of the population. We also conducted a simulation study based on an experiment in nanomaterials \citep{remillard2017tuning} and found that these results hold for realistic applications.

A recent strand of causal inference literature has elucidated and leveraged the connection between experimental design and finite-population sampling to refine theory and methodology for randomized experiments. This includes theory on design-based estimators for treatment effects \citep{samii2012equivalencies,aronow2013class}, properties of covariate-adjustment in randomized experiments \citep{freedman2008regression,lin2013agnostic,miratrix2013adjusting}, and methods for estimating treatment effects in complex experimental settings such as cluster-randomized experiments \citep{middleton2015unbiased}, experiments with interference \citep{aronow2013estimating}, and experiments with multiple treatments \citep{mukerjee2018using}. This work continues this trend of using experimental design and finite-population sampling techniques to characterize treatment effect estimation in randomized experiments. As we discussed in Section \ref{ss:clusterRandomizedDesigns}, a promising line for future work is to establish analogous results for other types of experiments, such as cluster-randomized designs, where there are multiple stages of stochasticity through sampling, randomization, and other mechanisms.

\begin{appendices}

\section{Proof of Lemma \ref{lem:S_properties}} \label{appendixA}

To prove (\ref{eq:S1}), note that
\begin{eqnarray*}
\E \left[S_{i}^{T_i} \right] = Pr \left[S_{i}^{T_i} = 1 \right] = Pr [ \text{Unit} \ i \ \text{is assigned treatment $T_i$ and is sampled}]
= N_{T_i}/N \times n_{T_i}/N_{T_i} = n_{T_i}/N.
\end{eqnarray*}

\noindent Next, we have that
\begin{eqnarray*}
\var \left[S_{i}^{T_i} \right] = \E \left[ \left( S_{i}^{T_i} \right)^2 \right] - \E^2 \left[S_{i}^{T_i} \right]
= n_{T_i}/N -  \left\{ n_{T_i}/N \right\}^2 = n_{T_i}/N \{ 1 - n_{T_i}/N \},
\end{eqnarray*}
which proves (\ref{eq:S2}).

To prove (\ref{eq:S4}), first we consider the case where $i=i^{\prime}$ and $T_i \ne T_{i^\prime}$:
\begin{eqnarray*}
\cov \left[ S_i^{0}, S_i^{1} \right] &=& \E \left[ S_i^{0} S_i^{1} \right] - \E \left[ S_i^{0}] \E[S_i^{1} \right] \\
&=& Pr \left[ \text{Unit} \ i \ \text{receives both treatments and is sampled} \right] - \E \left[ S_i^{0}] \E[S_i^{1} \right]  \\
&=& 0 - (n_0/N)(n_1/N) = - n_0 n_1 / N^2. 
\end{eqnarray*}

\noindent Next, for the case where $i \ne i^{\prime}$ and $T_i = T_{i^\prime}$, we have that:
\begin{eqnarray*}
\E \left[ S_i^{T_i} S_{i^\prime}^{T_i} \right] &=& Pr \left[ \text{Units} \ i, i^{\prime} \ \text{both receive treatment} \ T_i \ \text{and are sampled} \right]  \\
&=& Pr \left[ i, i^{\prime} \ \text{are both sampled} \ | \ \text{both receive} \ T_i \right] Pr \left[ \text{both receive} \ T_i \right] \\
&=& \frac{n_{T_i} (n_{T_i} - 1)}{N_{T_i} (N_{T_i} - 1)} \times \frac{N_{T_i} (N_{T_i} - 1)}{N(N-1)} = \frac{n_{T_i} (n_{T_i} - 1)}{N(N-1)}. 
\end{eqnarray*}

\noindent Thus, it follows that, 
\begin{eqnarray*}
\cov \left[ S_i^{T_i}, S_{i^{\prime}}^{T_{i}} \right] &=& \E \left[ S_i^{T_i} S_{i^\prime}^{T_i} \right] - \E \left[ S_i^{T_i} \right] \E \left[ S_{i^\prime}^{T_i} \right] = \frac{n_{T_i} (n_{T_i} - 1)}{N(N-1)} - \left(n_{T_i}/N \right)^2 \\
&=& - \frac{n_{T_i} \left\{ N - n_{T_i} \right\}}{N^2(N-1)},
\end{eqnarray*}
after a little algebra.

\noindent Finally, for the case where $i \ne i^{\prime}$ and $T_i \ne T_{i^\prime}$, we have that:
\begin{eqnarray*}
\E \left[ S_i^{0} \ S_{i^{\prime}}^{1} \right] &=& Pr \left[ i, i^{\prime} \ \text{sampled} | i \ \text{receives 0 and} \ i^{\prime} \ \text{receives 1} \right] Pr \left[ i \ \text{receives 0 and} \ i^{\prime} \ \text{receives 1} \right] \\
&=& Pr[ i \ \text{sampled} \ | \ i \ \text{receives} \ 0] Pr[ i^{\prime} \ \text{sampled} \ | \ i^{\prime} \ \text{receives} \ 1] Pr \left[ i \ \text{receives 0 and} \ i^{\prime} \ \text{receives 1} \right] \\
&=& \frac{n_0}{N_0} \times \frac{n_1}{N_1} \times \frac{N_0 N_1}{N(N-1)}.
\end{eqnarray*}
\noindent Consequently,
\begin{eqnarray*}
\cov \left[ S_i^{0}, \ S_{i^{\prime}}^{1} \right] &=& \E \left[ S_i^{0} \ S_{i^{\prime}}^{1} \right] - \E \left[ S_i^{0} \right] \E \left[ S_{i^{\prime}}^{1} \right] \\
&=& \frac{n_0}{N_0} \times \frac{n_1}{N_1} \times \frac{N_0 N_1}{N(N-1)} - \frac{n_0}{N} \times \frac{n_1}{N} 
= \frac{n_0 n_1}{N^2 (N-1)}.
\end{eqnarray*}

\section{Proof of Theorem \ref{thm:tauhat_properties}} \label{appendixB}

To prove the first part, note that for $T \in \{0,1\}$, $\bar{y}(T)$ defined in (\ref{eq:groupavg}) is an unbiased estimator of $\bar{Y}(T)$ by (\ref{eq:S1}). Consequently, from (\ref{eq:tauhat}), $\hat{\tau}$ is an unbiased estimator of $\tau$. To prove the second part, we need the following three lemmas:
\begin{lemma} \label{lem:theorem2_lem1}
For $T \in \{0,1\}$, $$ S^2_{T} = N^{-1} \left\{ \sum_{i=1}^N Y_i^2(T) - (N-1)^{-1} \underset{i \ne i^{\prime}}{\sum} Y_i(T) Y_{i^{\prime}}(T) \right\}.$$
\end{lemma}

\begin{proof}
By (\ref{eq:varT}), for $T \in \{0,1\}$,
\begin{eqnarray*}
(N-1) S^2_{T} &=&  \sum_{i=1}^N \{ Y_i(T) \}^2 - N^{-1} \left\{ \sum_{i=1}^N Y_i(T) \right\}^2 \\
&=& \sum_{i=1}^N \{ Y_i(T) \}^2 - N^{-1} \left[ \sum_{i=1}^N \{ Y_i(T) \}^2 + \underset{i \ne i^{\prime}}{\sum} Y_i(T) Y_{i^{\prime}}(T) \right] \\
&=& N^{-1}(N-1) \sum_{i=1}^N  \{ Y_i(T) \}^2 - N^{-1} \underset{i \ne i^{\prime}}{\sum} Y_i(T) Y_{i^{\prime}}(T) \\
&=& N^{-1}(N-1) \left[ \sum_{i=1}^N \{ Y_i(T) \}^2 - (N-1)^{-1} \underset{i \ne i^{\prime}}{\sum} Y_i(T) Y_{i^{\prime}}(T) \right].
\end{eqnarray*}
\end{proof}

\begin{lemma} \label{lem:theorem2_lem2}
For $T \in \{0,1\}$, $$ \var\left[ \bar{y}(T) \right] = \frac{N-n_{T}}{N} \cdot \frac{S^2_{T}}{n_{T}}.$$
\end{lemma}

\begin{proof}
By (\ref{eq:groupavg}), for $T \in \{0,1\}$,
\begin{eqnarray*}
\var\left[ \bar{y}(T) \right] &=&  \var\left[ (n_{T})^{-1} \sum_{i=1}^N S_i^{T} Y_i(T) \right]  \\
&=& (n_{T})^{-2} \left[ \sum_{i=1}^N  \{ Y_i(T) \}^2 \var \left[ S_i^{T} \right]  + \underset{i \ne i^{\prime}}{\sum} Y_i(T) Y_{i^{\prime}}(T) \cov \left[ S_i^{T}, S_{i^{\prime}}^{T} \right] \right] \\
&=& (n_{T})^{-2} \left[ \frac{n_{T}}{N} \left( 1 - \frac{n_{T}}{N} \right) \sum_{i=1}^N  \{ Y_i(T) \}^2 - \frac{n_{T}(N-n_{T})}{N^2 (N-1)} \underset{i \ne i^{\prime}}{\sum} Y_i(T) Y_{i^{\prime}}(T)  \right], \ \mbox{by Lemma} \ \ref{lem:S_properties} \\
&=& \frac{N - n_{T}}{n_{T} N^2} \left[ \sum_{i=1}^N  \{ Y_i(T) \}^2 - (N-1)^{-1} \underset{i \ne i^{\prime}}{\sum} Y_i(T) Y_{i^{\prime}}(T) \right] \\
&=& \frac{N-n_{T}}{N} \cdot \frac{S^2_{T}}{n_{T}}, \ \mbox{by Lemma} \ \ref{lem:theorem2_lem1}. 
\end{eqnarray*}
\end{proof}

\begin{lemma} \label{lem:theorem2_lem3} 
The covariance between $\bar{y}(1)$ and $\bar{y}(0)$ is given by:
$$ \cov \left[ \bar{y}(1), \bar{y}(0) \right ] = - \frac{S^2_1 + S^2_0 - S^2_{\tau}}{2N}.$$
\end{lemma}

\begin{proof}
\begin{eqnarray*}
\cov \left[ \bar{y}(1), \bar{y}(0) \right] &=& \cov \left[ n_1 ^{-1} \sum_{i=1}^N S_i^1 Y_i(1), n_0 ^{-1} \sum_{i=1}^N S_i^0 Y_i(0)  \right] \\
&=& (n_0 n_1)^{-1} \left[ \sum_{i=1}^N Y_i(1) Y_i(0) \cov \left[S_i^1, S_i^0 \right]  + \underset{i \ne i^{\prime}}{\sum} Y_i(1) Y_{i^{\prime}}(0) \cov \left[ S_i^{1}, S_{i^{\prime}}^{0} \right] \right] \\ 
&=& (n_0 n_1)^{-1} \left[  - \frac{n_0 n_1}{N^2} \sum_{i=1}^N Y_i(1) Y_i(0)  +  \frac{n_0 n_1}{N^2(N-1)} \underset{i \ne i^{\prime}}{\sum} Y_i(1) Y_{i^{\prime}}(0) \right], \ \mbox{by Lemma} \ \ref{lem:S_properties} \\ 
&=& - \frac{1}{N^2(N-1)} \left[  (N-1) \sum_{i=1}^N Y_i(1) Y_i(0)  -  \underset{i \ne i^{\prime}}{\sum} Y_i(1) Y_{i^{\prime}}(0) \right] \\
&=& - \frac{1}{N^2(N-1)} \left[  (N-1) \sum_{i=1}^N Y_i(1) Y_i(0)  - \left\{ \sum_{i=1}^N Y_i(1) \right\} \left\{ \sum_{i=1}^N Y_i(0) \right\} + \sum_{i=1}^N Y_i(1) Y_i(0) \right] \\
&=& - \frac{1}{N(N-1)} \left[ \sum_{i=1}^N Y_i(1) Y_i(0) - N \bar{Y}(1) \bar{Y}(0) \right] = - S_{10}/N, \ \mbox{by} \ (\ref{eq:cov10}).\\ 
\end{eqnarray*}
The result follows from the identity $S_{10} = \left(S^2_1 + S^2_0 - S^2_{\tau} \right)/2$, which is straightforward to establish.
\end{proof}

\noindent Straightforward applications of Lemma \ref{lem:theorem2_lem2} and Lemma \ref{lem:theorem2_lem3} lead to the proof of Theorem \ref{thm:tauhat_properties}.

\section{Deriving the Sampling Variance of $\widetilde{\tau}$} \label{appendixC}

Let $N$ denote the population size, $n$ denote the sample size, and $n_1$ and $n_0$ denote the sizes of treatment groups 1 and 0, respectively, all of which are fixed. Finally, let $S_i$ denote the sampling indicator, where $S_i = 1$ if unit $i$ is sampled to be in the completely randomized experiment, and $S_i = 0$ otherwise. Note that---following results from Lemma \ref{lem:S_properties} and from survey sampling (e.g., \citealt[Page 53]{lohr2009sampling})---we have that $\mathbb{E}[S_i] = n/N$, $\text{Var}(S_i) = \frac{n}{N} \left(1 - \frac{n}{N} \right)$, and $\text{Cov}(S_i, S_j) = - \frac{n(N-n)}{N^2(N-1)}$ for $i \neq j$. In \citet[Chapter 6, Appendix B]{imbens2015causal}, it was incorrectly stated that $\text{Cov}(S_i, S_j) = - \left( \frac{n}{N} \right)^2$, which is the main difference between their derivation and the derivation presented here.

The variance of $\widetilde{\tau}$ can be derived by conditioning on the sampling indicators, $\mathbf{S} = (S_1,\dots,S_N)$:
\begin{align}
	\text{Var}(\widetilde{\tau}) &= \mathbb{E}[ \text{Var}(\widetilde{\tau} | \mathbf{S})] + \text{Var} \left( \mathbb{E}[\widetilde{\tau} | \mathbf{S}] \right) \label{eqn:ateEVVEDecomposition}
\end{align}
As discussed in Section \ref{ss:comparison}, \cite{splawa1990application} showed that $\text{Var}(\widetilde{\tau} | \mathbf{S}) = \frac{s_1^2}{n_1} + \frac{s_0^2}{n_0} - \frac{s_{\tau}^2}{n}$, where $s_1^2$ and $s_0^2$ are the sample variances defined in (\ref{eq:sample_var}), and $s^2_{\tau} = \frac{1}{n-1} \sum_{i: S_i = 1} (\tau_i - \tau_S)^2$ is the variance of the treatment effects in the sample, where $\tau_S = n^{-1} \sum_{i: S_i = 1} \tau_i$ is the sample ATE. Meanwhile, $\mathbb{E}[\widetilde{\tau} | \mathbf{S}] = \tau_S$ \citep[Page 111]{imbens2015causal}. Finally, classical results on survey sampling (e.g., \citealt[Page 36]{lohr2009sampling}) have derived the variance of a sample mean under simple random sampling; these results give us $\text{Var}(\tau_S) = \frac{S^2_{\tau}}{n} \left(1 - \frac{n}{N} \right)$. Importantly, this last result is a consequence of the fact that $\text{Cov}(S_i, S_j) = - \frac{n(N-n)}{N^2(N-1)}$ for $i \neq j$.

Plugging these results back into (\ref{eqn:ateEVVEDecomposition}), we have
\begin{align}
	\text{Var}(\widetilde{\tau}) &= \mathbb{E} \left[ \frac{s_1^2}{n_1} + \frac{s_0^2}{n_0} - \frac{s_{\tau}^2}{n} \right] + \text{Var} \left( \tau_S \right) \\
	&= \frac{S_1^2}{n_1} + \frac{S_0^2}{n_0} - \frac{S_{\tau}^2}{n} + \text{Var} \left( \tau_S \right) \hspace{0.1 in} \text{(where $S_1^2$, $S_0^2$, and $S_{\tau}^2$ are defined in Section \ref{s:notations})} \\
	&= \frac{S_1^2}{n_1} + \frac{S_0^2}{n_0} - \frac{S_{\tau}^2}{n} + \frac{S^2_{\tau}}{n} \left(1 - \frac{n}{N} \right) \\
	&= \frac{S_1^2}{n_1} + \frac{S_0^2}{n_0} - \frac{S_{\tau}^2}{N}
	&= \text{Var}(\hat{\tau})
\end{align}
as mentioned in Section \ref{ss:comparison}, Equation (\ref{eqn:comparison}).

\section{Proof of Lemma \ref{lem:SZ_properties}} \label{appendixD}

We will consider the four cases in Lemma \ref{lem:SZ_properties}. For the first case, when $i = i'$ and $T_i = T_{i'}$,
\begin{equation}
\begin{aligned}
	\text{Cov}(S_i^{T_i}, Z_{i}^{T_{i}}) &= \mathbb{E}[S_i^{T_i} Z_{i}^{T_{i}}] - \mathbb{E}[S_i^{T_i}] \mathbb{E}[Z_{i}^{T_{i}}] \\
	&= Pr(S_i^{T_i} = 1, Z_{i}^{T_{i}} = 1) - \frac{n_{T_i} n_{x_{T_{i}}}}{N^2} \\
	&= Pr(S_i^{T_i} = 1 | Z_{i}^{T_{i}} = 1) Pr(Z_{i}^{T_{i}} = 1) - \frac{n_{T_i} n_{x_{T_{i}}}}{N^2} \\
	&= Pr(S_i^{T_i} = 1 | \text{unit $i$ is assigned to $T_i$}) Pr(Z_{i}^{T_{i}} = 1) - \frac{n_{T_i} n_{x_{T_{i}}}}{N^2} \\
	&= \frac{n_{T_i}}{N_{T_i}} \frac{n_{x_{T_{i}}}}{N} - \frac{n_{T_i} n_{x_{T_{i}}}}{N^2} \\
	&= \frac{N_{1-T_i}n_{T_i}n_{x_{T_i}}}{N^2 N_{T_i}}, \text{ after some algebra.}
	\label{eqn:lemma4Case1}
\end{aligned}
\end{equation}

Now for the second case. When $i = i'$ and $T_i \neq T_{i'}$, $Pr(S_i^{T_i} = 1, Z_{i'}^{T_{i'}} = 1) = 0$. Plugging this into the same procedure as (\ref{eqn:lemma4Case1}), we have that $\text{Cov}(S_i^{T_i}, Z_{i'}^{T_{i'}}) = - \frac{n_{T_i} n_{x_{T_{i'}}}}{N^2}$.

Now for the third case. When $i \neq i'$ and $T_i = T_{i'}$,
\begin{equation}
\begin{aligned}
	\text{Cov}(S_i^{T_i}, Z_{i'}^{T_{i}}) &= \mathbb{E}[S_i^{T_i} Z_{i'}^{T_{i}}] - \mathbb{E}[S_i^{T_i}] \mathbb{E}[Z_{i'}^{T_{i}}] \\
	&= Pr(S_i^{T_i} = 1, Z_{i'}^{T_{i}} = 1) - \frac{n_{T_i} n_{x_{T_{i}}}}{N^2} \\
	&= Pr(S_i^{T_i} = 1 | Z_{i'}^{T_{i}} = 1) Pr(Z_{i'}^{T_{i}} = 1) - \frac{n_{T_i} n_{x_{T_{i}}}}{N^2} \\
	&= Pr(S_i^{T_i} = 1 | \text{unit $i'$ is assigned to $T_i$}) Pr(Z_{i'}^{T_{i}} = 1) - \frac{n_{T_i} n_{x_{T_{i}}}}{N^2} \\
	&= \frac{n_{T_i} - 1}{N_{T_i} - 1} Pr(Z_{i'}^{T_{i}} = 1) - \frac{n_{T_i} n_{x_{T_{i}}}}{N^2} \\
	&= \frac{n_{T_i} - 1}{N_{T_i} - 1} \frac{n_{x_{T_i}}}{N} - \frac{n_{T_i} n_{x_{T_{i}}}}{N^2} \\
	&= - \frac{N_{1 - T_i} n_{T_i} n_{x_{T_i}}}{N^2(N-1)N_{T_i}}, \text{ after some algebra.}
	\label{eqn:lemma4Case3}
\end{aligned}
\end{equation}
Now for the fourth case. When $i \neq i'$ and $T_i \neq T_{i'}$,
\begin{equation}
\begin{aligned}
	\text{Cov}(S_i^{T_i}, Z_{i'}^{T_{i'}}) &= \mathbb{E}[S_i^{T_i} Z_{i'}^{T_{i'}}] - \mathbb{E}[S_i^{T_i}] \mathbb{E}[Z_{i'}^{T_{i'}}] \\
	&= Pr(S_i^{T_i} = 1, Z_{i'}^{T_{i'}} = 1) - \frac{n_{T_i} n_{x_{1-T_{i}}}}{N^2} \\
	&= Pr(S_i^{T_i} = 1 | Z_{i'}^{T_{i'}} = 1) Pr(Z_{i'}^{T_{i'}} = 1) - \frac{n_{T_i} n_{x_{1-T_{i}}}}{N^2} \\
	&= \frac{n_{T_i}}{N_{T_i}} \frac{N_{T_i}}{N-1} Pr(Z_{i'}^{T_{i'}} = 1) - \frac{n_{T_i} n_{x_{1-T_{i}}}}{N^2} \\
	&= \frac{n_{T_i}}{N_{T_i}} \frac{N_{T_i}}{N-1} \frac{n_{x_{1-T_i}}}{N} - \frac{n_{T_i} n_{x_{1-T_{i}}}}{N^2} \\
	&= \frac{n_{T_i} n_{x_{1-T_i}}}{N^2(N-1)}, \text{ after some algebra.}
	\label{eqn:lemma4Case4}
\end{aligned}
\end{equation}

\section{Proof of Theorem \ref{thm:tauhatstar_properties}} \label{appendixE}

To prove this theorem, we first state and prove the following two lemmas:

\begin{lemma} \label{lem:theorem3_lem1}
For $T, T^{\prime} \in \{0,1\}$, $\var[\bar{x}(T)] = \left(1 - n_{x_T}/N \right) S^2_x/n_{x_T}$ and
$\cov [\bar{x}(T),\bar{x}(T^{\prime})] = - S^2_x/N$.
\end{lemma}

\begin{proof}
The proof follows from arguments similar to those in the proofs of Lemma \ref{lem:theorem2_lem2} and Lemma \ref{lem:theorem2_lem3}, utilizing the properties of indicators $(Z_i^{T_i},S_i^{T_i})$ stated in Lemma \ref{lem:Z_properties} and \ref{lem:SZ_properties}.
\end{proof}

\begin{lemma} \label{lem:theorem3_lem2}
For $T, T^{\prime} \in \{0,1\}$,
\begin{align}
	\cov \left[ \bar{y}(T), \bar{x}(T^{\prime}) \right] = \begin{cases}
		\frac{N_{1-T}}{N_T} \frac{S_{Tx}}{N} &\mbox{ if } \hspace{0.1 in} T = T' \\
		- \frac{S_{Tx}}{N} &\mbox{ if } \hspace{0.1 in} T \neq T'
	\end{cases}
\end{align}
where $S_{Tx}$ is given by (\ref{eq:cov_yx}).
\end{lemma}

\begin{proof}
For the first case, when $T = T'$,
\begin{eqnarray*}
\cov \left[ \bar{y}(T), \bar{x}(T) \right]  &=& \cov \left[ n_T^{-1} \sum_{i=1}^N S_i^T Y_i(T), n_{x_T}^{-1} \sum_{i=1}^N Z_i^T X_i \right] \\
&=& (n_T n_{x_T})^{-1} \left[ \sum_{i=1}^N Y_i(T) X_i \cov \left(S_i^T, Z_i^T \right) + \underset{i \ne i^{\prime}}{\sum} Y_i(T) X_{i^{\prime}} \cov \left( S_i^T, Z_{i^\prime}^T \right) \right]
\end{eqnarray*}
Substituting expressions for $\cov \left(S_i^T, Z_i^T \right)$ and $\cov \left( S_i^T, Z_{i^\prime}^T \right)$ from Lemma \ref{lem:SZ_properties} and after some algebraic simplifications similar to those in Lemma \ref{lem:theorem2_lem3}, it follows that
\begin{eqnarray*}
\cov \left[ \bar{y}(T), \bar{x}(T) \right]
&=& \frac{1}{N^2 N_T (N-1)} \left[ N N_{1-T} \sum_{i=1}^N Y_i(T) X_i - N_{1-T} \left \{ \sum_{i=1}^N Y_i(T) \right\} \left \{ \sum_{i=1}^N X_i \right\} \right] \\
&=& \frac{N_{1-T}}{N_T} \frac{S_{Tx}}{N}
\end{eqnarray*} 
Next, for the second case where $T \neq T'$,
\begin{eqnarray*}
\cov \left[ \bar{y}(T), \bar{x}(T^{\prime}) \right]  &=& \cov \left[ n_T^{-1} \sum_{i=1}^N S_i^T Y_i(T), n_{x_{1-T}}^{-1} \sum_{i=1}^N Z_i^{T^{\prime}} X_i \right] \\
&=& (n_T n_{x_{1-T}})^{-1} \left[ \sum_{i=1}^N Y_i(T) X_i \cov \left(S_i^T, Z_i^{T^{\prime}} \right) + \underset{i \ne i^{\prime}}{\sum} Y_i(T) X_{i^{\prime}} \cov \left( S_i^T, Z_{i^\prime}^{T^{\prime}} \right) \right]
\end{eqnarray*}
Substituting expressions for $\cov \left(S_i^T, Z_i^{T^{\prime}} \right)$ and $\cov \left( S_i^T, Z_{i^\prime}^{T^{\prime}} \right)$ from Lemma \ref{lem:SZ_properties} and after some algebraic simplifications similar to those in Lemma \ref{lem:theorem2_lem3}, the above expression simplifies to $- S_{Tx}/N$. \\
 
\end{proof}

\noindent  Now to prove Theorem \ref{thm:tauhatstar_properties}, we write
\begin{equation}
\var(\hat{\tau}^*) = \var \left[ \hat{\tau} \right] + \var \left[ \bar{x}(1) - \bar{x}(0) \right] - 2 \cov \left[ \hat{\tau}, \bar{x}(1) - \bar{x}(0) \right]. \label{eq:thm3_expression}
\end{equation}
\noindent The second term equals $\var[\bar{x}(1)] + \var[\bar{x}(0)] - 2 \cov [\bar{x}(1), \bar{x}(0)]$. Substituting the expressions for $\var[\bar{x}(1)]$, $\var[\bar{x}(0)]$ and $\cov [\bar{x}(1), \bar{x}(0)]$ from Lemma \ref{lem:theorem3_lem1} and after some simplification, we have that
\begin{equation}
\var \left[ \bar{x}(1) - \bar{x}(0) \right] = \left(n_{x_1}^{-1} + n_{x_0}^{-1} \right) S^2_x.   \label{eq:thm3_term2}
\end{equation}

\noindent The expression $\cov \left[ \hat{\tau}, \bar{x}(1) - \bar{x}(0) \right]$ in the third term can be written as:
\begin{eqnarray}
&& \cov \left[ \bar{y}(1) - \bar{y}(0), \bar{x}(1) - \bar{x}(0) \right] \nonumber \\
&=& \cov \left[ \bar{y}(1), \bar{x}(1) \right] - \cov \left[ \bar{y}(1), \bar{x}(0) \right] - \cov \left[ \bar{y}(0), \bar{x}(1) \right] + \cov \left[ \bar{y}(0), \bar{x}(0) \right] \nonumber \\
&=& \frac{N_0}{N_1} \frac{S_{1x}}{N} + \frac{S_{1x}}{N} + \frac{S_{0x}}{N} + \frac{N_1}{N_0} \frac{S_{0x}}{N} \nonumber \\
&=& \frac{S_{1x}}{N_1} + \frac{S_{0x}}{N_0}, \label{eq:thm3_term3}
\end{eqnarray}
\noindent where the last step follows by substitution of the relevant expressions from Lemma \ref{lem:theorem3_lem2}. 

\noindent Finally, the result follows by substituting (\ref{eq:thm3_term2}) and (\ref{eq:thm3_term3}) in (\ref{eq:thm3_expression}).

\end{appendices}

\bibliographystyle{apalike}
\bibliography{samplingRandomizationBib}

\end{document}